\documentclass[a4paper,12pt]{amsart}

\usepackage{geometry}
\usepackage[utf8]{inputenc}
\usepackage[british]{babel}

\usepackage{amsthm,amssymb,amsmath,graphicx,mathtools, amsaddr}
\usepackage[shortlabels, inline]{enumitem}
\usepackage[abbrev]{amsrefs}

\allowdisplaybreaks
\newtheorem{theorem}{Theorem}

\newtheorem{claim}[theorem]{Claim}
\newtheorem*{claim*}{Claim}

\newtheorem{corollary}[theorem]{Corollary}

\newtheorem{lemma}[theorem]{Lemma}

\newtheorem{proposition}[theorem]{Proposition}

\numberwithin{equation}{section}
\numberwithin{theorem}{section}
\numberwithin{case}{section}

\numberwithin{subcase}{case}

\theoremstyle{definition}
\newtheorem{algonice}[theorem]{Algorithm}

\def\eps{\varepsilon}
\def\NN{\mathbb{N}}

\def\tildadeg{\rho}
\def\dfn{:=}

\newenvironment{proofclaim}[1][Proof of the claim]{\begin{proof}[#1]}{\end{proof}}

\setenumerate[1]{label=\textnormal{(\roman*)}}
\setenumerate[2]{label=\textnormal{(\alph*)}}

\usepackage[usenames,dvipsnames]{color}

\usepackage{hyperref}

\title{Density of monochromatic infinite paths} 
\author{Allan Lo and Nicolás Sanhueza-Matamala}
\stepcounter{g@author}
\address[Allan Lo, Nicolás Sanhueza-Matamala]{School of Mathematics,\\ University of Birmingham,\\ Birmingham B15 2TT,\\ United Kingdom}
\email{s.a.lo@bham.ac.uk, NIS564@bham.ac.uk}
\author{Guanghui Wang}
\address[Guinghui Wang]{School of Mathematics, \\ Shandong University, \\ Jinan, 250100, China}
\email{ghwang@sdu.edu.cn}
\thanks{The research leading to these results was partially supported by EPSRC, grant no. EP/P002420/1 (A.~Lo), the Becas Chile scholarship scheme from CONICYT (N.~Sanhueza-Matamala) and NSFC No. 11471193, 11631014 (G.~Wang).}
	
\begin{document}	

\begin{abstract}
	For any subset $A \subseteq \mathbb{N}$, we define its upper density  to be $\limsup_{ n \rightarrow \infty } |A \cap \{ 1, \dotsc, n \}| / n$.
	We prove that every $2$-edge-colouring of the complete graph on $\mathbb{N}$ contains a monochromatic infinite path, whose vertex set has upper density at least $(9 + \sqrt{17})/16 \approx 0.82019$.
	This improves on results of Erd\H{o}s and Galvin, and of DeBiasio and McKenney.
\end{abstract}

\maketitle

\section{Introduction}

A $2$-edge-colouring of a graph~$G$ is an assignment of 2 colours, red and blue, to each edge of~$G$. 
We say that $G$ is \emph{monochromatic} if all the edges of~$G$ are coloured with the same colour.
Given an arbitrary $2$-edge-colouring of $K_n$, what is the size of the largest monochromatic path contained as a subgraph?
This was answered by Gerencs\'er and Gy\'arf\'as~\cite{GerencserGyarfas1967}, who proved that every $2$-edge-coloured~$K_n$ contains a monochromatic path of length at least~$2n/3$.
This result is sharp. 

Now consider the infinite complete graph~$K_{\mathbb{N}}$ on the vertex set $\mathbb{N}$.
For any subset $A \subseteq \mathbb{N}$, the \emph{upper density} $\overline{d}(A)$ of $A$ is defined as \[ \overline{d}(A) \dfn \limsup_{n \rightarrow \infty} \frac{|A \cap \{ 1, \dotsc, n \}|}{n}. \]
Given a subgraph~$H$ of~$K_{\mathbb{N}}$, we define the \emph{upper density $\overline{d}(H)$ of~$H$} to be that of $V(H)$.
Trying to generalise the results known in the finite case, it is natural to ask what are the densest paths which can be found in any $2$-edge-coloured $K_{\mathbb{N}}$.
This problem was considered first by Erd\H{o}s and Galvin~\cite{ErdosGalvin1993}.
Other variants of this problem have been studied as well.
For example, it is possible to consider other monochromatic subgraphs rather than paths, edge-colourings with more than two colours, use different notions of density or consider monochromatic sub-digraphs of infinite edge-coloured digraphs, etc.
Results along these lines have been obtained by Erd\H{o}s and Galvin~\cite{ErdosGalvin1991, ErdosGalvin1993}, DeBiasio and McKenney~\cite{DeBiasioMcKenney2016} and Bürger, DeBiasio, Guggiari and Pitz~\cite{Guggiari2017}.

We focus on the case of monochromatic paths in $2$-edge-coloured complete graphs.
By a classical result of Ramsey Theory, any $2$-edge-colouring of~$K_{\mathbb{N}}$ contains a monochromatic infinite complete graph, and therefore, also a monochromatic infinite path~$P$.
However, this argument alone cannot guarantee a monochromatic path with positive upper density, as it was shown by Erd\H{o}s~\cite{Erdos1964} that there exist $2$-edge-colourings of the infinite complete graph where every infinite monochromatic complete subgraph has upper density zero.
Rado~\cite{Rado1978} showed that in every $r$-edge-coloured $K_{\mathbb{N}}$ there are $r$ monochromatic paths, of distinct colours, which partition the vertex set.
This immediately implies that every $2$-edge-coloured $K_\mathbb{N}$ contains an infinite monochromatic path $P$ with $\overline{d}(P) \ge 1/2$.

Erd\H{o}s and Galvin~\cite{ErdosGalvin1993} proved that for every $2$-edge-colouring of $K_{\mathbb{N}}$ there exists a monochromatic path $P$ with $\overline{d}(P) \ge 2/3$ and showed an example of a $2$-edge-colouring of $K_{\mathbb{N}}$ such that every monochromatic path satisfies $\overline{d}(P) \leq 8/9$.
DeBiasio and McKenney~\cite{DeBiasioMcKenney2016} improved the lower bound and showed that for every $2$-edge-colouring of $K_{\mathbb{N}}$, there exists a monochromatic path $P$ with $\overline{d}(P) \ge 3/4$.
In this paper, we improve the lower bound on $\overline{d}(P)$. 

\begin{theorem} \label{theorem:main}
	Every $2$-edge-colouring of $K_{\mathbb{N}}$ contains a monochromatic path $P$ with $\overline{d}(P) \ge (9 + \sqrt{17})/16 \approx 0.82019$.
\end{theorem}

In Section~\ref{section:pathforest} we state our main lemma (Lemma~\ref{lemma:nopurplevertices}) and use it to deduce Theorem~\ref{theorem:main}.
In Section~\ref{section:tools} we collect some useful tools that will be used during the proof of Lemma~\ref{lemma:nopurplevertices}, which is done in Section~\ref{section:algorithm}.

\subsection{Notation}

Given a graph $G$, we write $V(G)$ and $E(G)$ for its vertex and edge set, respectively; and $e(G) \dfn |E(G)|$.
Given $S \subseteq V(G)$, we write $G[S]$ for the subgraph of $G$ induced by $S$.
If $S, T \subseteq V(G)$ are disjoint, we write $G[S,T]$ for the bipartite graph with classes~$S$ and~$T$ consisting precisely of those edges in $G$ with one endpoint in $S$ and the other in $T$.

Let $G$ be a $2$-edge-coloured graph. 
Throughout the paper, we assume its colours to be red and blue. % (which will often referred by $R$ and $B$, respectively). 
For a vertex $x \in V(G)$ and a subset $S \subseteq V(G)$, we write the \emph{red neighbourhood of~$x$ in~$S$} for the set $N^R_G(x, S) \dfn \{ y \in S : xy \text{ is coloured red} \}$, that is, the set of vertices in $S$ connected to $x$ with red edges.
We define $N^B_G(x, S)$ analogously for blue. 
For all $\ast \in \{ R, B \}$, we also define $d^\ast_G(x, S) \dfn |N^\ast_G(x, S)|$ whenever $N^\ast_G(x, S)$ is finite, $d^\ast_G(x, S) \dfn \infty$ otherwise.

For every $i \ge 0$, let $[i] \dfn \{ 1, \dotsc, i \}$ and $[i]_0 \dfn [i] \cup \{ 0 \}$.
For every set $S \subseteq \mathbb{N}$ and $t \in \mathbb{N}$ we write $S \cup t$ for $S \cup \{t\}$.

We write $x \ll y$ to mean that for all $y \in (0, 1]$ there exists $x_0 \in (0,1)$ such that for all $x \leq x_0$ the following statements hold.
Hierarchies with more constants are defined in a similar way and are to be read from right to left.

\section{Monochromatic path-forests} \label{section:pathforest}

Our proof follows the strategies of Erd\H{os} and Galvin~\cite{ErdosGalvin1993} and of DeBiasio and McKenney~\cite{DeBiasioMcKenney2016}, where they reduce the problem of finding monochromatic paths to the problem of finding collections of monochromatic disjoint paths satisfying certain conditions, which are then joined together to form an infinite path.

Consider a $2$-edge-coloured~$K_{\mathbb{N}}$.
We say a vertex $x \in \mathbb{N}$ is \emph{red} (or \emph{blue}) if $x$ has infinitely many red (or blue, respectively) neighbours in $K_{\mathbb{N}}$.
Note that it is possible for a vertex to be both red and blue.
A $2$-edge-colouring of $K_{\mathbb{N}}$ is \emph{restricted} if there is no vertex that is both red and blue.
We write $R$ and $B$ for the set of red and blue vertices of $K_{\NN}$, respectively.

A \emph{path-forest} is a collection of vertex-disjoint paths.
Let $K_{\mathbb{N}}$ be a $2$-edge-coloured graph.
A path-forest~$F$ of~$K_{\mathbb{N}}$ is said to be \emph{red} if every edge of~$F$ is red and all endpoints of every path in $F$ are red.
We further assume that, for every path~$P$ in~$F$, its vertices~$V(P)$ alternate between red and blue.
Note that a red path-forest may contain isolated red vertices. 
A \emph{blue path-forest} is defined similarly.

Our main lemma states that given a restricted $2$-edge-coloured~$K_{\mathbb{N}}$, there exists a monochromatic path-forest~$F$ and an arbitrary long interval~$[t]$ such that $V(F) \cap [t]$ has size which is linear in $t$.

\begin{lemma} \label{lemma:nopurplevertices}
	Let $\eps \in (0, 1/2)$ and $k_0 \in \NN$.
%	Let $\alpha = (7-\sqrt{17})/16$.
	For every restricted $2$-edge-coloured~$K_{\NN}$, there exists an integer $t \ge k_0$ and red and blue path-forests $F^R$ and~$F^B$, respectively, such that 
	\begin{align*}
	\max \{ |V(F^R) \cap [t]|, |V(F^B) \cap [t]| \} \ge ((9 + \sqrt{17})/16 - \eps) t. 
	\end{align*}
\end{lemma}

We defer the proof of Lemma~\ref{lemma:nopurplevertices} to Section~\ref{section:algorithm}.
Note that we can always add any vertex which is both red and blue to a monochromatic path-forest, as an isolated vertex.
Thus Lemma~\ref{lemma:nopurplevertices} implies the following corollary, which is valid for arbitrary $2$-edge-colourings.

\begin{corollary} \label{corollary:densempf}
Let $\eps \in (0, 1/2)$ and $k_0 \in \NN$.
For every $2$-edge-coloured~$K_{\NN}$, there exists an integer $t \ge k_0$ and red and blue path-forests $F^R$ and~$F^B$, respectively, such that 
	\begin{align*}
	\max \{ |V(F^R) \cap [t]|, |V(F^B) \cap [t]| \} \ge ((9 + \sqrt{17})/16 - \eps) t. 
	\end{align*}
\end{corollary}

%\begin{proof}
%	Let $W$ be the set of vertices which are simultaneously red and blue under the canonical vertex-colouring of $K_{\mathbb{N}}$.
%		Suppose first that $\mathbb{N} \setminus W$ is finite.
%		Then for $t \ge k$ large enough, $|W \cap [t]| \ge ( 1 - \eps )t$.
%		The vertices of $W \cap [t]$ form a monochromatic red path-forest $F$ with $|V(F) \cap t| \ge (1 - \eps) t \ge ( (9 + \sqrt{17})/16 - \eps )t$, as desired.
%		
%		Hence, we might suppose that $\mathbb{N} \setminus W$ is infinite.
%		Suppose $\mathbb{N} \setminus W = \{ v_1, v_2, \dotsc \}$ where $v_i < v_j$ for all $i < j$.
%		Consider the induced subgraph of $K_{\mathbb{N}}$ on $\mathbb{N} \setminus W$, together with the inherited edge-colouring.
%		Note that the canonical vertex-colouring in $\mathbb{N} \setminus W$ corresponds exactly to the restriction of the original canonical vertex-colouring to $\mathbb{N} \setminus W$.
%		In particular, the canonical vertex-colouring in $\mathbb{N} \setminus W$ is restricted.
%		Then Lemma~\ref{lemma:nopurplevertices} implies the existence of a (say) red path-forest $F$ in $K_{\mathbb{N}}$ with $V(F) \cap W = \varnothing $ and $t \ge k$ such that $|V(F) \cap \{ v_1, \dotsc, v_t \}| \ge ((9 + \sqrt{17})/16 - \eps) t$.
%		Then $F \cup (W \cap [v_t] )$ is a red path-forest in $K_{\mathbb{N}}$ with \begin{align*}
%			|V(F) \cap [v_t]| & = |W \cap [v_t]| + |V(F) \cap [v_t]| \ge v_t - t + ((9 + \sqrt{17})/16 - \eps)t \\
%			& \ge ((9 + \sqrt{17})/16 - \eps)v_t,
%		\end{align*} as desired.
%\end{proof}

We use it now to deduce Theorem~\ref{theorem:main}.
The proof is based on the proofs of \cite[Theorem 3.5]{ErdosGalvin1993} and \cite[Theorem 1.6]{DeBiasioMcKenney2016}.

\begin{proof}[Proof of Theorem~\ref{theorem:main}]
Consider an arbitrary $2$-edge-colouring of $K_{\mathbb{N}}$.
Suppose that there exist two red vertices $x_1, x_2 \in \mathbb{N}$ and a finite subset~$S$ of~$\mathbb{N}$ such that $K_{\mathbb{N}} \setminus S$ does not contain a red path between $x_1$ and~$x_2$.
For $i \in [2]$, let $X_i$ be the set of vertices reachable from $x_i$ using red paths in~$\mathbb{N} \setminus S$.
Let $X_3 = \mathbb{N} \setminus (X_1 \cup X_2 \cup S)$.
Then $X_1$ and $X_2$ are infinite; $X_1, X_2$ and $X_3$ are pairwise disjoint and there are no red edges between any $X_i, X_j$ for distinct $i,j \in [3]$.
Thus there is an infinite blue path~$P$ on the vertex set $X_1 \cup X_2 \cup X_3 = \mathbb{N} \setminus S$.
Since~$S$ is finite, $\overline{d}(P) = 1$, so we are done.
An analogous argument is true if red is swapped with blue.
Hence, we might assume that
\begin{align}
	\parbox{0.85\textwidth}{for any two red (or blue) vertices $x_1,x_2$ and any finite set~$S \subseteq \mathbb{N} \setminus \{ x_1,x_2\}$, there is a red (or blue, respectively) path joining $x_1$ and $x_2$ in $K_{\mathbb{N}} \setminus S$.} \label{eq:joinability}
\end{align}
	
For all $i \in\mathbb{N}$, let $\eps_i \dfn 1 / (2i) $.
If the vertex $1$ is red, set $P^R_1 = (\{1\}, \varnothing)$ to be the red path with the vertex~$1$ and $P^B_1$ to be empty.
Otherwise, set $P^R_1$ to be empty and $P^B_1 = (\{1\}, \varnothing)$.
Set $n_1 = 1$.
Suppose that, for some $i \in \mathbb{N}$, we have already found an integer $n_{i}$ and red and blue paths~$P_{i}^R$ and $P_{i}^B$, respectively, such that the endpoints of $P_i^R$ are red, the endpoints of $P_i^B$ are blue; and
\begin{align}
	\label{eqn:P^R_i}
	\max\{ |V(P^R_i) \cap [n_i]|, |V(P^B_i) \cap [n_i]| \} \ge ((9 + \sqrt{17})/16 - 2 \eps_i)n_i.
\end{align}
We construct $n_{i+1}$, $P_{i+1}^R$ and $P_{i+1}^B$ as follows.
Let $r_i \dfn \max \{ V(P^R_i) , V(P^B_i), n_i\}$ and $k_i \dfn r_i/ \eps_{i+1} = 2(i+1) r_i$.
Considering the induced subgraph of $K_{\mathbb{N}}$ on $\mathbb{N} \setminus [r_i]$, by Corollary~\ref{corollary:densempf}, there exists a monochromatic path-forest $F_{i+1}$ and $t_i \ge k_i$ such that $ | V ( F_{i+1} ) \cap \{ r_i+1, \dotsc, r_i+t_i \} | \ge ( (9 + \sqrt{17})/16 - \eps_{i+1}  ) t_i$.
Let $n_{i+1} \dfn r_i+t_i$.
By the choice of $k_i$, note that
\begin{align*}
	|V(F_{i+1}) \cap [n_{i+1}] |
	& \ge ( (9 + \sqrt{17})/16 - \eps_{i+1}  ) t_i 
	 \ge ( (9 + \sqrt{17})/16 - 2 \eps_{i+1}  ) n_{i+1}.
\end{align*}	
Suppose $F_{i+1}$ is red (if not, interchange the colours in what follows).
Let $P^B_{i+1} \dfn P^B_i$.
Apply~\eqref{eq:joinability} repeatedly to join the endpoints of the paths in~$P^R_i \cup F_i$ and obtain a red path $P^R_{i+1}$ containing $P^R_{i}$ and~$F_i$ with red vertices as endpoints.

By construction, we have $n_{i+1} >n_i$ and \eqref{eqn:P^R_i} holds for all $i \ge 1$.
Without loss of generality, we may assume that $ |V(P^R_i) \cap [n_i]| \ge ((9 + \sqrt{17})/16 - 2 \eps_i)n_i$ for infinitely many values of~$i$. 
Let $P \dfn \bigcup_{i \ge 1} P^R_i$.
Therefore, $P$ is a monochromatic path and $\overline{d}(P) \ge (9 + \sqrt{17})/16 $.
\end{proof}

\section{Preliminaries} \label{section:tools}

In this section, we consider two ways of extending a path forest. 

\begin{proposition} \label{proposition:extendpathforest}
	Let $G$ be a graph.
	Let $F \subseteq G$ be a path-forest and let $J \subseteq V(F)$ be the set of vertices with degree at most one in $F$.
	Let $x \in V(G) \setminus V(F)$ be such that $d_G(x, J) \ge 3$.
	Then there exist~$j_1, j_2 \in V(F)$ such that $F \cup \{ xj_1, x j_2 \}$ is a path-forest.
\end{proposition}

\begin{proof}
	Since $d_G(x, J) \ge 3$, there exist at least two neighbours of~$x$ in~$J$, which are not endpoints of the same path in~$F$.
\end{proof}

\begin{proposition} \label{proposition:usingavailabledegree}
	Let $G$ be a graph and $F \subseteq G$ a path-forest.
	Let $Y \subseteq V(G) \setminus V(F)$ and $X \subseteq V(F)$. 
	Suppose that \begin{enumerate}
		\item \label{item:x0x1size} $\sum_{x \in X} (2- d_F(x)) \ge 2 |Y|$, and
		\item \label{item:x0x1degree} for every $x \in X$, $d_G(x, Y) \ge |Y| - 2$.
	\end{enumerate} 
	Then there exists a path-forest $F' \subseteq G[X, Y]$; every path in~$F'$ has both endpoints in $X$; $F \cup F'$ is a path-forest and $|V(F') \cap Y| \ge |Y| - 4$.
\end{proposition}

\begin{proof}
Without loss of generality, we may assume that $d_{F}(x) < 2$ for all $x \in X$.
We proceed by induction on~$|Y|$.
It is trivial if $|Y| \le 4$ (by setting $F'$ to be empty).
So we may assume that $|Y| \ge 5$. 
Note that $|X| \ge 5$ by~\ref{item:x0x1size}.
Pick $x_1, x_2 \in X$ be such that $x_1$ and $x_2$ are not connected in~$F$.
By~\ref{item:x0x1degree} and $|Y| \ge 5$, there exists $y \in Y \cap N_G(x_1) \cap N_G(x_2)$.
Set~$F_1 \dfn F \cup \{ x_1 y, x_2 y \}$ and $Y' \dfn Y \setminus \{ y \}$.
It is easy to check that $F_1,X,Y'$ also satisfy the corresponding \ref{item:x0x1size} and~\ref{item:x0x1degree}.
Therefore, by our induction hypothesis, the proposition holds. 
\end{proof}

The next lemma is a useful statement about difference inequalities.
We include its proof for completeness.

\begin{lemma} \label{lemma:differenceinequality}
	Let $\tau_1, \tau_2 > 0$, $c_0 \ge 0$ be given and let $s_0, s_1, \dotsc$ be a strictly increasing sequence of non-negative integers.
	Suppose there exists $n_0$ such that for every $n \ge n_0$, \[s_{n+1} \leq \tau_1 s_n - \tau_2 s_{n-1} + c_0.\]
	Then $\tau^2_1 \ge 4 \tau_2$.
\end{lemma}

\begin{proof}
	Suppose $\tau^2_1 < 4 \tau_2$.
	Choose $\delta \in (0,1)$ sufficiently small such that $\tau_1^2 < 4 \tau_2 (1 - \delta)$ and let $\rho_1 \dfn \tau_1/(1 - \delta)$ and $\rho_2 \dfn \tau_2/(1 - \delta)$.
	Since $\{ s_n \}_{n \in \mathbb{N}}$ is a strictly increasing sequence of non-negative integers, there exists $n_1 \ge n_0$ such that \begin{align*}
	\delta s_n \ge c_0\text{ for every $n \ge n_1$.}
	\end{align*}
	Then, for $n \ge n_1$, $s_{n+1} \leq \tau_1 s_n - \tau_2 s_{n-1} + \delta s_{n+1}$, which implies, for every $n \ge n_1$,
	\begin{align}
	s_{n+1} \leq \rho_1 s_n - \rho_2 s_{n-1}. \label{eq:constrainedgrowth}
	\end{align}
	
	Consider the function $f : (- \infty, \rho_1) \rightarrow \mathbb{R}$ given by $f(x) = \rho_2/(\rho_1-x)$.
	It is immediate that $f$ is continuous.
	Since $\rho^2_1 < 4 \rho_2$, it follows that $x < f(x)$ for all $x < \rho_1$.
	
	For every $n \ge n_1$, let $\beta_n \dfn s_{n+1}/s_n$.
	From \eqref{eq:constrainedgrowth}, for every $n \ge n_1$, \[ 1 < \beta_n < \rho_1. \] Using \eqref{eq:constrainedgrowth} it also follows that $ \rho_1 s_n - \rho_2 s_{n-1} \ge \beta_n s_n, $ which can be rearranged to get \[ \beta_{n-1} = \frac{s_n}{s_{n-1}} \ge \frac{\rho_2}{\rho_1 - \beta_n} = f( \beta_n ) > \beta_n. \]
	Since $\beta_n$ is monotone decreasing and bounded, it converges to a limit $\beta \in [1, \rho_1)$.
	Moreover, the sequence $f(\beta_n)$ converges to the same limit.
	The continuity of $f$ implies that $\beta = f(\beta) > \beta$, a contradiction.
\end{proof}

\section{Proof of Lemma~\ref{lemma:nopurplevertices}} \label{section:algorithm}

\subsection{The path-forests algorithm}

To satisfy the conditions stated in Lemma~\ref{lemma:nopurplevertices}, we consider an algorithm that will build path-forests considering one extra vertex at a time, in increasing order. 

Our algorithm is based on the following simple idea. 
Suppose that $t \in \mathbb{N}$ is a red vertex and we have constructed red and blue path-forests $F^R$ and $F^B$, respectively. 
We can add~$t$ to~$F^R$ without any difficulty, forming a new red path-forest.
We would like to add~$t$ to the blue path-forest~$F^B$ as well.
However, we will add $t$ to the blue path-forest~$F^B$ using only forward edges or only backward edges.
Namely, when we say ``add~$t$ to~$F^B$ using forward edges'' (or backward edges) we mean to add the blue edges $t j_1, t j_2$ to~$F^B$ for some blue vertices $j_1,j_2 >t$ (or $j_1,j_2 < t$, respectively).
We remark that the red (or blue) path-forest will contain all the red (or blue) vertices that have been considered so far, but it might be possible that some vertices are never included in the path-forest of the opposite colour.
 
Here we give an outline of Algorithm~\ref{algo}.
There is a positive even integer~$\ell$ which will be chosen before running the algorithm.
The algorithm will consider each $t \in \mathbb{N}$ in order to decide whether to add it to the path-forest of the opposite colour by using forward or backward edges, with a preference toward forward edges. 
In fact, the algorithm will add a vertex using forward edges straight away, if possible, but will only add vertices using backward edges in batches.
Roughly speaking, $A^R_t$ will be an (ordered) set of red vertices $v \in [t]$ such that $v$ is joined to almost all blue vertices $w > v$ with red edges. 
Once $A^R_t$ is large enough, we will set aside a subset~$\Omega^R$ of~$A^R_t$ ``of size $\ell$'', which will be the `smaller' endpoints of the backward edges. 
We continue the algorithm and collect a set $\Gamma^B$ of blue vertices, which could not be included in the red path-forest by using red forward edges. 
Once $\Gamma^B$ has $\ell$ vertices, we then add most of the vertices of $\Gamma^B$ into the red path-forest using red backward edges between $\Omega^R$ and~$\Gamma^B$.

During the course of the algorithm, we will also construct a function $\varphi : \mathbb{N} \rightarrow \mathbb{N}$, which will help us to define the sets $A^R_t, A^B_t$ at any given step.
The role of $\varphi$ is the following: a red vertex $t$ will be part of $A^R_{t'}$ only when $t' \ge \varphi(t)$, similarly with the blue vertices.
Imprecisely speaking, for a red vertex we would like $\varphi(t)$ to be ``the last'' of the blue vertices connected to $t$ via forward blue edges (this makes sense since the colouring is restricted); if no such blue vertices exist we just define $\phi(t) = t$.
If $t' = \varphi(t)$ is chosen like this, then when the algorithm reaches step $t'$, the red vertex $t$ will now be connected to ``most'' of the upcoming blue vertices using only red edges, which makes $t$ suitable to belong in $A^R_{t'}$.

Before presenting the algorithm, we will need the following notation. 
Suppose that after round number $t$, we have constructed red and blue path-forests $F^R_t$ and $F^B_t$, respectively.
Given an ordered vertex set~$V = \{v_i \colon i \in [n]\}$ and $\ast \in \{R,B\}$, define \[ \tildadeg^\ast_t(V) \dfn \sum_{v \in V} ( 2 - d_{F^\ast_t}(v)). \]
We view $\tildadeg^\ast_t(V)$ to be the number of additional degree that we can (theoretically) add to $V$ while keeping $F^\ast_t$ being a path-forest.
Suppose an even $\ell \in \mathbb{N}$ is given and $V = \{ v_i : i \in [n] \}$.
If $\tildadeg^\ast_t(V) \ge \ell$, then we define $\sigma^\ast_t(V)$ in the following way:
let $s \in [n]$ be minimal such that $\tildadeg^\ast_t(\{v_i \colon i \in [s]\}) \ge \ell$ and then select $V' \subseteq \{v_i \colon i \in [s]\} \subseteq V$ to be minimal with respect to inclusion such that $\tildadeg^\ast_t(V') \ge \ell$; and let $\sigma^\ast_t(V) := V'$.
Note that, by choice, $d_{F^\ast_t}(v) \leq 1$ for all $v \in V'$.
%Let $\sigma^{F, \ell}(V) = \{v \in \hat{\sigma}^{F, \ell}(V) \colon d_F(v) \leq 1\}$, that is, the subset of $\hat{\sigma}^{F, \ell}(V)$ that has positive contribution to $\tildadeg_F(\hat{\sigma}^{F, \ell}(V)  )$.
Note as well that $\tildadeg^\ast_t(\sigma^{\ast}_t(V)) \in \{ \ell, \ell+1 \}$.
(Referring to the outline above, we will set $\Omega^R = \sigma^\ast_t(A^R_t)$.)
%For~$\ast \in \{R,B\}$, we write $c_t^* \dfn |V(F_t^\ast) \cap [t]|$.% $\tildadeg_t^*(V) \dfn \tildadeg_{F_t^*}(V)$ and $\sigma_t^*(V) \dfn \sigma^{F_t^*, \ell}(V)$.

We make the following crucial definition.
For all $\ast \in \{R,B\}$ and $t \in \mathbb{N}$, we define
\begin{align*}
c^\ast_t & := |V(F^\ast_t) \cap [t]|.
\end{align*}

We are now ready to describe the algorithm. 
We will verify that this algorithm is well-defined in Lemma~\ref{lemma:algovalid}.

\begin{algonice} \label{algo}
Fix an even $\ell \in \mathbb{N}$.
Given any restricted $2$-edge-colouring of $K_{\NN}$, we now construct monochromatic path-forests as follows.
Initially, let $F^\ast_{0}, A^{\ast}_0, \Omega^\ast_0, \Gamma^\ast_0, \varphi_0$ be empty for all $\ast \in \{ R, B \}$.
Now suppose that we are at round number $t \ge 1$, and we have already constructed monochromatic path-forests~$F^{\ast}_{t-1}$, an ordered vertex subset $A^{\ast}_{t-1}$, vertex subsets~$\Omega^{\ast}_{t-1}, \Gamma^{\ast}_{t-1}$ for $\ast \in \{R,B\}$ and a function~$\varphi_{t-1} : [t-1] \rightarrow \NN$.

We now construct $F^{\ast}_{t}, A^{\ast}_{t}, \Omega^{\ast}_{t}, \Gamma^{\ast}_{t}, \varphi_t$ as follows by considering the vertex~$t \in \mathbb{N}$.
Suppose~$t \in R$ (and if $t \in B$, interchange the roles of $R$ and~$B$ in what follows).
Our algorithm works in four steps.

\begin{enumerate}[label=\textbf{Step \arabic*:}, ref=\arabic*, wide, labelwidth=!, labelindent=0pt]
	\item \textbf{Adding $t$ to the red path-forest}. \\
	Set $F^R_{t} \dfn F^R_{t-1} \cup  t $.

	\item \textbf{Updating available and waiting blue vertices}.\\
	Let $A^B_t$ be obtained from $A^{B}_{t-1}$ by adding the vertices $v \in [t-1]$ with $\varphi_{t-1}(v) = t$ at the end of the ordering and $\Gamma_t^B \dfn \Gamma_{t-1}^B$.
	If $\tildadeg_{t-1}^B ( A^B_{t}) \ge \ell$ and $\Omega^B_{t-1} = \varnothing$, then set $\Omega_{t}^B \dfn \sigma_{t-1}^B(A_{t}^B)$; otherwise set $\Omega^B_t \dfn \Omega^B_{t-1}$.
	
	\item \label{step:classifyingt} \textbf{Classifying~$t$}.\\
	We now classify~$t$ into one of four types, which will use to determine whether (and how) $t$ can be added to the blue path-forest $F^B_{t-1}$.
	Let $J \dfn \{ v \in  N^B_{K_{\mathbb{N}}}(t,  B \setminus [t] ) \colon d_{F^B_{t-1}}(v) < 2\}$.
	That is, $J$ is the blue neighbourhood of~$t$, that theoretically we can use to attach~$t$ to $F_{t-1}^B$ using blue forward edges without creating a vertex of degree~$3$.
	If $\Omega^{B}_{t} \ne \varnothing$, then we set $t_{\Omega}$ to be the smallest $t_{\Omega}$ such that $\Omega^{B}_{t_{\Omega}} = \Omega^{B}_{t}$.
	We say that $t$ is 
	\begin{itemize}
		\item \emph{of type~$W$} if $|J| \ge 3$;
		\item \emph{of type~$X$} if $|J| \le 2$ and $\Omega^{B}_{t} = \varnothing$;
		\item \emph{of type~$Y$} if $|J| \le 2$, $\Omega^{B}_{t} \ne \varnothing$ and $d_{F_{t_{\Omega}}^R}(t) <2$;
		\item \emph{of type~$Z$} if $|J| \le 2$, $\Omega^{B}_{t} \ne \varnothing$ and $d_{F_{t_{\Omega}}^R}(t) =2$.
	\end{itemize}
	
	\item \label{step:tbluepf} \textbf{Trying to add $t$ to the blue path-forest}.\\
	Depending on the type of $t$, we have three different cases.
	
	\begin{enumerate}[label=\textbf{Step 4\alph*:}, ref=4\alph*, wide, labelwidth=!, labelindent=0pt]
		\item \label{step:typeW} \textbf{$t$ is of type~$W$}.\\
		We add $t$ to $F^B_{t}$ using forward edges.
		By Proposition~\ref{proposition:extendpathforest} (with $F^B_{t-1}, J, t$  playing the roles of $F, J, x$) there exist $j_1, j_2 \in J$ such that $F^B_{t-1} \cup \{ t j_1, t j_2 \}$ is a blue path-forest.
		Further choose $j_1$ and $j_2$ such that $\min \{ j_1, j_2 \}$ is maximised (which is well-defined as $t \in R$ and the colouring is restricted, so $J \subseteq N^B_{K_{\mathbb{N}}}(t)$ is finite).		
		Define $\varphi_{t}(t) = \min \{ j_1, j_2 \}$ and $\varphi_{t}(i) = \varphi_{t-1}(i)$ for all $i \in [t-1]$.
		Set $F_{t}^B \dfn F^B_{t-1} \cup \{ t j_1, tj_2 \}$, $A_t^R \dfn A_{t-1}^R$, $\Omega^R_t \dfn \Omega^R_{t-1}$ and $\Gamma_t^R \dfn \Gamma_{t-1}^R$.
		
		\item \label{step:typeXZ} \textbf{$t$ is of type~$X$ or $Z$}.\\
		In this case, we will not add~$t$ to~$F^B_{t-1}$ at all. 
		Define $\varphi_{t}(t) = t$ and $\varphi_{t}(i) = \varphi_{t-1}(i)$ for all $i \in [t-1]$.
		Set $F_{t}^B \dfn F^B_{t-1}$.
		Let~$A_t^R$ be obtained from $A_{t-1}^R$ by adding $t$ to the end of the ordering.
		If $\tildadeg_{t}^R ( A^R_{t}) \ge \ell$ and $\Omega^R_{t-1} \dfn \varnothing$, set $\Omega_{t}^R \dfn \sigma_{t}^R(A_{t}^R)$; otherwise set $\Omega^R_t \dfn \Omega^R_{t-1}$.
		Finally, set $\Gamma_t^R \dfn \Gamma_{t-1}^R$.
		
		\item \label{step:typeY} \textbf{$t$ is of type~$Y$}.\\
		In this case, we will try to add $t$ to $F^B_t$ using backwards edges if $\Gamma^R_t$ has reached the correct size.
		Define $\varphi_t$, $A^R_t$ and $\Omega^R_t$ as in Step~\ref{step:typeXZ}.

		If $|\Gamma_{t-1}^R \cup t| < \ell/2$, then set~$F_{t}^B \dfn F_{t-1}^B$ and $\Gamma_{t}^R \dfn \Gamma_{t-1}^R \cup t$ and finish this step.
		Otherwise, we have $|\Gamma_{t-1}^R \cup t| = \ell/2$.
		By Proposition~\ref{proposition:usingavailabledegree} (with $F_{t-1}^B , \Omega^{B}_{t}, \Gamma_{t-1}^R \cup t$ playing the roles of~$F, X, Y$), we obtain a blue path-forest $F'$ such that~$F_{t-1}^B \cup F'$ is a blue path-forest which covers all but at most $4$ vertices of $\Gamma_{t-1}^R \cup t$.
		Let~$F_{t}^B \dfn F^B_{t-1} \cup F'$.
		Adding the new blue edges to form $F^B_t$ means we need to redefine $\Omega^B_t$ accordingly, as follows: if~$\tildadeg_{t}^B ( A^B_{t}) \ge \ell$, then redefine $\Omega_{t}^B \dfn \sigma_{t}^B(A_{t}^B)$; otherwise redefine $\Omega^B_t \dfn \varnothing$.
		Finally, define $\Gamma_t^R \dfn \varnothing$.
	\end{enumerate}
\end{enumerate}

\end{algonice}

\subsection{Correctness and analysis of the algorithm}

First we show that Algorithm~\ref{algo} is well-defined.
For $t \in \mathbb{N}$, define $W_{t}^R$ (and $W_t^B$) to be the set of vertices $v \in [t] \cap R$ (and $v \in [t] \cap B$, respectively) of type~$W$, as in Step~\ref{step:classifyingt} of Algorithm~\ref{algo}.
Similarly, define $X_t^{\ast}, Y_t^{\ast}, Z_t^{\ast}$ for $\ast \in \{ R , B\}$.

\begin{lemma} \label{lemma:algovalid}
Let $\ell \in \mathbb{N}$ be even. 
Then Algorithm~\ref{algo} is well defined. 
\end{lemma}

\begin{proof}
Suppose that $K_{\mathbb{N}}$ has a restricted $2$-edge-colouring.
We prove by induction on~$t$ that $F_t^{\ast}$, $\Omega_t^{\ast}$, $\Gamma_t^{\ast}$, $A_t^{\ast}$, $\varphi_t$, $W_t^{\ast}$, $X_t^{\ast}$, $Y_t^{\ast}$, $Z_t^{\ast}$ given by Algorithm~\ref{algo} satisfy the following properties (and similar statements hold if we interchange $R$ and $B$):
\begin{enumerate}[label={\rm(\roman*)}]
	\item \label{itm:correctnessfirsteasy} $\varphi_t(i) \ge i$ for all $i \in [t]$ and $\varphi_t(i) = \varphi_{t-1}(i)$ for all $i \in [t-1]$;
	\item if $i \in R \cap [t]$ and $\varphi_t(i) > i$, then $\varphi_t(i) \in B$;
	\item $A^R_t, \Omega^R_t, \Gamma^R_t, \subseteq R \cap [t]$, $\Omega^R_t \subseteq A^R_t$ and $A^R_{t-1} \subseteq A^R_t$;
	\item \label{itm:phioftheavailable} $\{\varphi_t(v) \colon v \in A^R_t\} \subseteq [t]$;
	\item if $\Omega^R_t \neq \varnothing$, then 
	$\tildadeg^R_t( \Omega^R_t ) \in \{ \ell, \ell + 1 \}$;
	\item $|\Gamma^R_t| < \ell/2$;
	\item \label{itm:correctnesslasteasy} if $y > t$ and $y \in R$, then $y \notin V(F^B_t)$;
	\item \label{itm:correctnessnoteasy} if $\Omega^R_t, \Gamma^B_t \neq \varnothing$, then  $\max \Omega^R_t \le \max\{ \varphi_t(v) \colon v \in \Omega^R_t\} < \min \Gamma^B_t$ and
	for all $v \in \Omega^R_t$, $d^B_{K_{\mathbb{N}}}(v, \Gamma^B_t) \leq 2$.
\end{enumerate}
Note that these properties imply the lemma.
By our construction, \ref{itm:correctnessfirsteasy}--\ref{itm:correctnesslasteasy} hold.

To see \ref{itm:correctnessnoteasy}, let $t_{\Omega}$ to be the smallest $t_{\Omega}$ such that $\Omega^{R}_{t_{\Omega}} = \Omega^{R}_{t}$.
%By (iv), we have $\varphi_t(v) = \varphi_{t_\Omega}(v) \le t_\Omega$ for all $v \in \Omega^B_r$
Consider any $v \in \Omega^R_t$.
Clearly $v \le \varphi_t(v) \le t_{\Omega} \le \min \Gamma^B_t$ by~\ref{itm:correctnessfirsteasy} and \ref{itm:phioftheavailable}.
So the first assertion of~\ref{itm:correctnessnoteasy} holds.
Let $J \dfn \{ j' \in N^B_{K_{\mathbb{N}}}(v, B \setminus [v])  \colon d_{F^B_{v-1}}(j') < 2\}$, which is $J$ defined at round number $v$.
For all $u \in \Gamma^B_t \subseteq Y^B_t$, we have $d_{F^B_{v-1}}(u) \le d_{F^B_{t_\Omega}}(u) < 2$.
Hence $\Gamma^B_t \subseteq J$.
If $v$ is not of type~$X$, then $d^B_{K_{\mathbb{N}}}(v, \Gamma^B_t) \le d^B_{K_{\mathbb{N}}} (v, J) \leq 2$.
If $v$ is of type~$X$, then $d^B_{K_{\mathbb{N}}} (v, \Gamma^B_t) \ge 3$ would contradict the maximality of~$\varphi_t(v)$ in Step~\ref{step:typeW}.
Hence we have $d^B_{K_{\mathbb{N}}}(v, \Gamma^B_t) \leq 2$ for all $v \in \Omega^R_t$.
\end{proof}

Recall that for every $\ast \in \{R,B\}$ and $t \in \NN$, $c^\ast_t  = |V(F^\ast_t) \cap [t]|$.
In the next two lemmas, we collect some useful information from the algorithm.

\begin{lemma} \label{lemma:usefulfacts}
Let $\ell \in \NN$ be even.
Suppose that $K_{\mathbb{N}}$ has a restricted $2$-edge-colouring.
Let $F_t^{\ast}$, $\Omega_t^{\ast}$, $\Gamma_t^{\ast}$, $A_t^{\ast}$, $\varphi_t$, $W_t^{\ast}$, $X_t^{\ast}$, $Y_t^{\ast}$, $Z_t^{\ast}$ be as defined by Algorithm~\ref{algo}.
Then the following holds for all $t \in \mathbb{N}$ (and similar statements hold if we interchange $R$ and $B$):
\begin{enumerate}[label={\rm(\roman*)}]
	\item \label{itm:u0} $|R \cap [t]| = |W^{R}_t|+|X^{R}_t|+|Y^{R}_t|+|Z^{R}_t|$;
	\item \label{itm:u0'} $F^R_t, W^{R}_t,X^{R}_t,Y^{R}_t,Z^{R}_t$ are nested;
	%\item \label{itm:u1} $V(F^R_t) \cap [t] \subseteq V(F^R_{t'}) \cap [t]$ for all $t \le t'$;
%	\item \label{itm:u8a} $\{\varphi_t(v) \colon v \in A^R_t\} \subseteq [t]$; \COMMENT{already mentioned in Lemma~\ref{lemma:algovalid}}
	\item \label{itm:u3} if there exists $t' \ge t$ such that $\Omega^B_{t''} \ne \varnothing$ for all $t\le t'' \le t'$, then $X^R_t = X^R_{t'}$;
%		\COMMENT{For the case if $\rho >\ell$, then $X = \varnothing$.}
	\item \label{itm:u5} if $v \in V(F^R_t)$ with $v >t$, then $v \in R$ and $N_{F^{R}_t} (v) \subseteq W^B_t$;
	\item \label{itm:u5'} if $v \in B $ with $d_{F^R_t}(v) >0$, then $v \in W^B_t \cup Y^B_t$;
	\item \label{itm:u6} if $\rho_t^R(A_t^R) \ge \ell$, then $\Omega^R_t \neq \varnothing$;
	%\item \label{itm:u6'} if $\Omega^R_t \neq \varnothing$, then $\max \Omega^R_t \le \max\{ \varphi_t(v) \colon v \in \Omega^R_t\} < \min \Gamma^B_t$
		\item \label{itm:u2} $c^R_{t} \ge 	(1 - 8 /\ell ) ( t - |Z^B_{t}| - |X_{t}^B| ) - \ell/2$;
%		\COMMENT{Easier to calculate $c_R^t$.}
		\item \label{itm:u7} $2 |Y^B_{t'} \setminus Y^B_t| \ge \rho_{t}^R (A^R_t) - \rho_{t'}^R (A^R_t)$ for $t' \ge t$;
	\item \label{itm:u8} if $\rho_{t'-1}^R(A_t^R) \ge \ell$ for some $t' \ge t$, then $|Z^B_{t'}| \le |W^R_t|$.
\end{enumerate}
\end{lemma}

\begin{proof}
Note that \ref{itm:u0}--\ref{itm:u6} hold by our construction.

Now we prove \ref{itm:u2}.
By our construction, we have $W^B_t, R \cap [t] \subseteq V(F^R_t)$.
%Suppose that $|Y^B_t| \ge \ell /2$.
Partition $Y^B_t$ into $\Gamma'_1, \Gamma'_2, \dots, \Gamma'_s, \Gamma'_{s+1}$ (with $\Gamma'_{s+1}$ possibly empty) such that, for all $i \in [s]$, $|\Gamma'_i| = \ell /2$, $\max \Gamma'_i < \min \Gamma'_{i+1}$ and $|\Gamma'_{s+1}| < \ell/2$.
In other words, $\Gamma'_1, \Gamma'_2, \dots, \Gamma'_s, \Gamma'_{s+1}$ is a partition of $Y^B_t$ into sets of `consecutive' $\ell /2$ vertices.
Consider any $i \in [s]$.
Let $t_i \dfn \max \Gamma'_{i}$.
Since $t_i \in Y_{t_i}^B$, Step~\ref{step:typeY} implies that we have $|\Gamma_{t_i-1}^B| = \ell/2 - 1$, $\Gamma_{t_i-1}^B \cup t_i = \Gamma'_{i}$ and $\Gamma_{t_i}^B = \varnothing$.
Moreover, all but at most $4$ vertices of~$\Gamma'_{i}$ are added to~$F^R_t$ (at round number $t_i$).
Therefore, 
\begin{align*}
	c^R_{t} & = |V(F^R_t) \cap [t] |  \ge |R \cap [t]| + |W^B_t| + \sum_{i \in [s]} (|\Gamma'_{i}|-4) \\
	& = |R \cap [t]| + |W^B_t| + \sum_{i \in [s]} (1- 8/\ell) |\Gamma'_{i}| \\
	& \ge |R \cap [t]| + |W^B_t| +  (1 - 8/\ell )(|Y^B_t| - \ell/2 )\\
	& \ge (1 - 8/\ell ) ( t - |X^B_t| - |Z^B_t| ) - \ell/2.
\end{align*}
Hence \ref{itm:u2} holds.

To see~\ref{itm:u7}, note that $\rho_{t''}^R (A^R_t)$ is a decreasing sequence in $t''$ and it decreases if and only if we join some vertices of $A^R_t$ to some vertices in $Y^B_{t'} \setminus Y^B_{t}$ with red edges to form the red path-forest. 
Each such vertex of $y \in Y^B_{t'} \setminus Y^B_{t}$ reduces $\rho_{t}^R (A^R_t)$ by at most~$2$.

To see~\ref{itm:u8}, since $\rho_{t'-1}^R(A_t^R) \ge \ell$ for some $t' \ge t$, we have $\Omega_{t''}^R \subseteq A_t^R$ for all $t \le t'' < t'$.
Note that 
\begin{align*}
	\max_{v \in \Omega_{t''}^R }\{\varphi_{t''} (v) \}
	& \le \max_{v \in A_t^R }\{\varphi_{t''} (v) \} \le t.
\end{align*}
Consider any $z \in Z^B_{t'}$.
By Step~\ref{step:classifyingt} of Algorithm~\ref{algo}, this means that $d_{F_{t}^B}(z) =2$. 
Hence $d_{F_{t}^B}(z) = 2$ for all $z \in  Z^B_{t'}$.
By~\ref{itm:u5}, $N_{F^{B}_t} (z) \subseteq W^R_t$ for all $z \in Z^B_{t'}$.
By counting the number of edges in $F^{B}_t [ Z^B_{t'},  W^R_t]$, we have 
\begin{align*}
	2 |Z^B_{t'}| = e( F^{B}_t [ Z^B_{t'},  W^R_t] ) \le 2 |W^R_t|
\end{align*}
implying~\ref{itm:u8}.
\end{proof}

\begin{lemma} \label{lemma:rho}
Let $\ell \in \NN$ be even.
Suppose that $K_{\mathbb{N}}$ has a restricted $2$-edge-colouring.
For all $t \in \mathbb{N}$, let $F_t^{\ast}, \Omega_t^{\ast}, \Gamma_t^{\ast}, A_t^{\ast}, \varphi_t, W_t^{\ast}, X_t^{\ast}, Y_t^{\ast},Z_t^{\ast}$ be as defined by Algorithm~\ref{algo}.
Then there exist $Y_t^{\ast} \subseteq D^{\ast}_t \subseteq  W_t^{\ast} \cup Y_t^{\ast}$ for all $t \in \mathbb{N}$ and~$\ast \in \{R,B\}$ such that (where similar statements hold if we interchange $R$ and $B$): 
\begin{enumerate}[label={\rm(\roman*)}]
	\item \label{itm:rho1} $\rho_{t'}^R (A^R_{t'}) - \rho_{t}^R (A^R_t) \leq 2 | D_{t'}^R \setminus D_{t}^R | + 2 | X_{t'}^R \setminus X_{t}^R |$, for every $t' \ge t$;
	\item \label{itm:rho2} $2 |D_t^B| \ge 2 |D_t^R \cup X_t^R \cup Z_t^R| - \rho_{t}^R (A_{t}^R)$;
	\item \label{itm:rho2strong} if $\rho^B_{t'-1}(A^B_t) \ge \ell$ for some $t' \ge t$, then $2 |D_t^B| \ge 2 |D_t^R| + |X_{t'}^R \cup Z_{t'}^R| - \rho_{t}^R (A_{t}^R)$;
	\item \label{itm:rho3} $c_t^R + c_t^B + \frac{1}{2} \rho_{t}^R (A_{t}^R) + \frac{1}{2}\rho_{t}^B (A_{t}^B)  \ge 2( 1- 8 /\ell)t - \ell$.
\end{enumerate}
\end{lemma}

\begin{proof}
Let $U_t^R \dfn \{ w \in W_t^R \colon \varphi_t (w) \le t \}$.
Let $D_t^R \dfn U_t^R \cup Y_t^R$.
Note that $A_t^R  = D_t^R \cup X_t^R \cup Z_t^R$ (here we view $A_t^R$ as an unordered set).
Hence
\begin{align}
	\label{eqn:rho1}
	\rho_t^R(A_t^R)  = \rho_t^R (D_t^R) + \rho_t^R (X_t^R) +\rho_t^R (Z_t^R)
	= \rho_t^R (D_t^R) + \rho_t^R (X_t^R)
	.
\end{align}
as $d_{ F^R_{ t } } (z) = 2$ for all $z \in Z_{t}^R$.
Note that $ U_t^R \subseteq U_{t'}^R$ for~$t < t'$.
Hence
\begin{align*}
	\rho_{t'}^R (A_{t'}^R) & = \rho_{t'}^R (D_{t'}^R) + \rho_{t'}^R (X_{t'}^R) \\
	& = \rho_{t'}^R ( D_{t'}^R \setminus D_{t}^R ) + \rho_{t'}^R (X_{t'}^R \setminus X_{t}^R)  + \rho_{t'}^R (A_{t}^R)\\
	& \le  2 | D_{t'}^R \setminus D_{t}^R | + 2 | X_{t'}^R \setminus X_{t}^R |  + \rho_{t}^R (A_{t}^R)
\end{align*}
implying~\ref{itm:rho1}.

Let $G^R_t \dfn F^R_t[ \{ 1, \dotsc, t \} ]$.
Since $F^R_t$ is a red path-forest, $G^R_t$ is a bipartite graph with vertex classes $R' \subseteq R \cap [t]$ and $B' \subseteq B \cap [t]$.
If $v \in B $ with $d_{F^R_t}(v) >0$, then $v \in W^B_t \cup Y^B_t$ by Lemma~\ref{lemma:usefulfacts}\ref{itm:u5'}.
If $v \in W^B_t $ with $d_{F^R_t}(v) >0$, then we must have $\varphi_t(v) \le t$ and so $v \in U^B_t$.
Hence if $v \in B $ with $d_{F^R_t}(v) >0$, then $v \in  D^B_t$.
Therefore, 
\begin{align}
	\label{eq:degreered}
	e(G^R_t) \leq 2 |D^B_t|. 
\end{align}
On the other hand, since $V(F^R_t) \cap R = R \cap [t] = V(G^R_t) \cap R$,
\begin{align*}
	e(G^R_t) & = \sum_{u \in R \cap [t]} d_{G^R_t}(u) = \sum_{u \in R \cap [t]} d_{F^R_t}(u)\\
	& \ge \sum_{u \in D_t^R \cup X_t^R \cup Z_t^R} d_{F^R_t}(u) 
	= 2 |D_t^R \cup X_t^R \cup Z_t^R| - \rho_{t}^R (A_{t}^R).
\end{align*}
Together with~\eqref{eq:degreered}, we obtain~\ref{itm:rho2}.

To see \ref{itm:rho2strong} proceed similarly but considering the graph $F^R_t[ \{ 1, \dotsc, t \} \cup X^R_{t'} \cup Z^R_{t'} ]$.
Lemma~\ref{lemma:usefulfacts}\ref{itm:u6} and \ref{itm:u3} imply that $X^R_{t'} \setminus X^R_{t} = \varnothing$; together with Lemma~\ref{lemma:usefulfacts}\ref{itm:u5} it implies that for every $u \in Z^R_{t'}$, $N_{F^R_t}(u) \subseteq W^B_t$ and $d_{F^R_{t}}(u) = 2$.
Counting the edges of $F^R_t[  \{ 1, \dotsc, t \} \cup X^R_{t'} \cup Z^R_{t'}  ]$ in two different ways, as before, gives the desired inequality.

By adding~\ref{itm:rho2} and its analogus version, we get
\begin{align}
	\frac12\rho_{t}^R (A_{t}^R) + \frac12 \rho_{t}^B (A_{t}^B) & \ge | X_t^R \cup Z_t^R \cup X_t^B \cup Z_t^B|.
	\label{eqn:rho2}
\end{align}
Lemma~\ref{lemma:usefulfacts}\ref{itm:u2} implies that 
\begin{align*}
	c_t^{R} + c_t^B & \ge 2(1- 8 /\ell) t - | X_t^R \cup Z_t^R \cup X_t^B \cup Z_t^B| - \ell,
\end{align*}
which together with~\eqref{eqn:rho2} implies~\ref{itm:rho3}.
\end{proof}

\subsection{Evolutions of $\rho^R_t(A^R_t)$ and $\rho^B_t(A^B_t)$}
To prove Lemma~\ref{lemma:nopurplevertices}, we will consider the path-forests $F^R_t$, $F^B_t$ for every $t \ge 1$, as constructed by Algorithm~\ref{algo}.
If, given $\eps$ and~$k_0$, for some $t \ge k_0$ we have $\max \{ c^R_t, c^B_t \} \ge ((9 + \sqrt{17})/16 - \eps) t$, then we are done.
Therefore, assuming this is not the case, we will deduce information about the evolution of the parameters $\rho^R_t(A^R_t)$ and $\rho^B_t(A^B_t)$ whenever $t$ increases, which we will use to finish the proof.
(It also suffices to use Lemmas~\ref{lemma:usefulfacts} and~\ref{lemma:rho} instead of appealing to Algorithm~\ref{algo}.)

First, we show that if $\rho_{t}^B ( A_t^B ) \ge \ell$ then there exists $t' > t$ such that $\rho_{t'}^B( A_t^B ) < \ell$ (or we are already done).
That is, almost all vertices $A_t^B$ have degree~$2$ in the red path-forest at round number $t'$. 

\begin{lemma} \label{lemma:emptytheavailable}
Let $\ell \in \NN$ be even.
Suppose that $K_{\mathbb{N}}$ has a restricted $2$-edge-colouring.
Let $F_t^{\ast}, \Omega_t^{\ast}, \Gamma_t^{\ast}, A_t^{\ast}, \varphi_t, W_t^{\ast}, X_t^{\ast}, Y_t^{\ast},Z_t^{\ast}$ be as defined by Algorithm~\ref{algo}.
Suppose $\rho_{t}^B ( A_t^B ) \ge \ell$.
Then there exists $t' > t$ such that $\rho_{t'}^B ( A_t^B ) < \ell$ or $c^B_{t'} \ge (1- 9/\ell)t'$.
\end{lemma}

\begin{proof}
Suppose that $\rho_{t'}^B ( A_t^B ) \ge \ell$ for all $t' > t$ (or else we are done).
By Lemma~\ref{lemma:usefulfacts}\ref{itm:u6}, $\Omega^B_{t'} \neq \varnothing$ for all $t' > t$.
Hence $X_{t'}^R = X_t^R$ for all $t' >t $ by Lemma~\ref{lemma:usefulfacts}\ref{itm:u3}.
%Let $t_{\max} \dfn \max\{ \varphi_{t}(v) \colon v \in A_t^B \}$.
%Note that $t_{\max} \ge \max\{ \varphi_{t'}(v) \colon v \in \Omega^B_{t'}, t'  \ge t\}$ as $\Omega^B_{t'} \subseteq A_t^B$ for all $t' \ge t$.
Moreover, Lemma~\ref{lemma:usefulfacts}\ref{itm:u8} implies that $|Z^R_{t'}| \le |W^B_t|$ for all $t' \ge t$.
Let $t' = \ell ( t + \ell / 2 )$.
Lemma~\ref{lemma:usefulfacts}\ref{itm:u2} implies that 
\begin{align*}
	c^B_{t'}
	& \ge (1 - 8 /\ell ) t' - |Z^R_{t'}| - |X_{t'}^R|  - \ell/2	\\
	& \ge (1 - 8 /\ell ) t' - |W^B_{t}| - |X_{t}^R|  - \ell/2\\
	& \ge (1 - 8 /\ell ) t' - ( t  + \ell/2) 
	\ge ( 1- 9 /\ell ) t'. \qedhere
\end{align*}
\end{proof}

\begin{lemma} \label{lemma:nomoreavailable}
Let $\ell \in \NN$ be even and $1/t_{0}\ll 1 / \ell \ll \eps \leq 1/2$.
Suppose that $K_{\mathbb{N}}$ has a restricted $2$-edge-colouring.
Let $F_t^{\ast}$, $\Omega_t^{\ast}$, $\Gamma_t^{\ast}$, $A_t^{\ast}$, $\varphi_t$, $W_t^{\ast}$, $X_t^{\ast}$, $Y_t^{\ast}$, $Z_t^{\ast}$ be as defined by Algorithm~\ref{algo}.
Suppose that $ \rho_{t_0}^B(A_{t_0}^B) \ge \ell$.
Then there exists $t' > t_0$ such that $ \rho_{t'}^B(A_{t'}^B)  < \ell$ or $\max\{ c^R_{t'}, c^B_{t'}\} \ge (2 \sqrt 2 - 2 - \eps) t'$.
\end{lemma}

\begin{proof}
Let $\alpha \dfn 3 - 2 \sqrt{2}$. % and note that $1 - \alpha = 2 \sqrt{2} - 2$.
Suppose the contrary, that is, for all $t > t_0$ we have
\begin{align}
	\rho_{t}^B(A_{t}^B) \ge \ell \text{ and } c^R_{t}, c^B_{t} \leq (1 - \alpha - \eps) t.
	\label{eqn:av1}
\end{align}
Note that Lemma~\ref{lemma:usefulfacts}\ref{itm:u3} and~\ref{itm:u6} imply that 
\begin{align}
	X^{R}_t = X^{R}_{t_0} 
	\label{eqn:Xi}
\end{align}
for all $t \ge t_0$.

Given~$t_i$, define $t^R_{i+1}$ to be the minimum $t > t_i$ such that $\rho_t^R(A_{t_i}^R) < \ell$, which exists by Lemma~\ref{lemma:emptytheavailable} and $1/\ell \ll \eps \leq 1/2$.
Analogously, define~$t^B_{i+1}$.
Define $t_{i+1} \dfn \max \{ t^R_{i+1}, t^B_{i+1} \}$ and $t'_{i+1} \dfn \min \{ t^R_{i+1}, t^B_{i+1} \}$.
This defines sequences $t_i,t_i'$ such that, for all $i \ge 1$,
\begin{align*}
		t_{i-1} & < t'_i \le t_i, \\
		\min \{ \rho_{t'_{i}}^R ( A_{t_{i-1}}^R ), \rho_{t'_{i}}^B ( A_{t_{i-1}}^B) \}, & \rho_{t_{i}}^R ( A_{t_{i-1}}^R ), \rho_{t_{i}}^B ( A_{t_{i-1}}^B) < \ell.
\end{align*}
For convenience, let $t_{-1} \dfn 0$ and for every $i \ge 0$, let $I_i \dfn \{ t_{i-1} + 1, \dotsc, t_i \}$.
For every $i \ge 0$ and $\ast \in \{R, B\}$, let
\begin{align*}
	x_i^{\ast} & \dfn |I_i \cap X_{t_i}^{\ast}| \text{ and }
	z_i^{\ast} \dfn |I_i \cap Z_{t_i}^{\ast}|.
\end{align*}
Lemma~\ref{lemma:usefulfacts}\ref{itm:u2} and \eqref{eqn:av1} imply that
\begin{align*}
	(1 - \alpha - \eps) t_i
	& \ge c^R_{t_i} \ge (1 - 8 /\ell ) t_i - |Z^B_{t_i}| - |X_{t_i}^B|  - \ell/2 \\
	& %\overset{\mathclap{\text{\eqref{eqn:Xi}}}}\ge
	\ge (1 - 8 /\ell ) t_i - \sum_{j \in [i]_0}(x_j^B+z_j^B) - \ell/2,
\end{align*}
and a similar inequality also holds for $\sum_{j \in [i]_0}(x_j^R+z_j^R)$.
In summary, we have for $\ast \in \{R,B\}$,
\begin{align}	
	\sum_{j \in [i]_0} (x_j^*+z_j^*) & \ge  (\alpha +\eps/2)  t_i. \label{eqn:sumq}
\end{align}

Consider any $i \ge 1$. 
Write $T_i \dfn \sum_{j \in [i]_0} t_{j}$.
Lemma~\ref{lemma:usefulfacts}\ref{itm:u7} implies that 
\begin{align*}
	|Y^B_{t_i} \setminus Y^B_{t_{i-1}}| & \ge |Y^B_{t^R_{i}} \setminus Y^B_{t_{i-1}}|
	\ge \frac12 ( \rho_{t_{i-1}}^R (A^R_{t_{i-1}}) - \rho_{t^R_{i}}^R (A^R_{t_{i-1}}) )
	\ge \frac12 ( \rho_{t_{i-1}}^R (A^R_{t_{i-1}}) - \ell) 
\end{align*}
and a similar inequality holds for $|Y^R_{t_i} \setminus Y^R_{t_{i-1}}|$.
Hence by combining both inequalities and using Lemma~\ref{lemma:rho}\ref{itm:rho3}, we have 
\begin{align*}
	|Y^B_{t_i} \setminus Y^B_{t_{i-1}}| + |Y^R_{t_i} \setminus Y^R_{t_{i-1}}| & \ge \frac12 ( \rho_{t_{i-1}}^R (A^R_{t_{i-1}}) +  \rho_{t_{i-1}}^B (A^B_{t_{i-1}}))- \ell \\
	& \ge 2( 1- 8 /\ell)t_{i-1} - 2 \ell - c^R_{t_{i-1}} - c^B_{t_{i-1}} \\
	& \overset{\mathclap{\text{\eqref{eqn:av1}}}}\ge 2(\alpha + \eps) t_{i-1} - 2 \ell - 16 t_{i-1} / \ell \\
	& \ge 2 (\alpha +\eps/2)t_{i-1},
\end{align*}
where the last inequality follows from $1/{t_{i-1}} \leq 1/t_0 \ll 1/\ell \ll \eps$.
Hence, for all $i \ge 0$,
\begin{align}
	|Y^B_{t_i} \cup Y^R_{t_i} | \ge 2 (\alpha +\eps/2) T_{i-1}. \label{eqn:YB+YR}
\end{align}

\begin{claim} \label{claim:termsaiai1}
For all $i \in \mathbb{N}$, $|W^R_{t_i} \cup W^B_{t_i} \cup X^B_{t_i} \cup Z^B_{t_i} | \ge (\alpha +\eps/2)(t_i + t_{i+1}) -t_0$.
\end{claim}

\begin{proofclaim}
We divide the proof into two cases.
First suppose that $t^{B}_{i+1} \ge t^{R}_{i+1}$.
Since $\rho_{t_{i}-1}^B ( A_{t_{i-1}}^B ) \ge \ell$, Lemma~\ref{lemma:usefulfacts}\ref{itm:u8} implies that $ |W_{t_i}^B| \ge |Z_{t_{i+1}}^R| = \sum_{j \in [i+1]_0} z^R_j$.
Hence 
\begin{align*}
	& |W^R_{t_i} \cup W^B_{t_i} \cup X^B_{t_i} \cup Z^B_{t_i} | \\
	& \ge | W^B_{t_i}| + |X^B_{t_i} \cup Z^B_{t_i} |
	\ge \sum_{j \in [i+1]_0 } z^R_j + \sum_{j\in [i]_0} (x^B_j + z^B_j) \\
	& \overset{\mathclap{\text{\eqref{eqn:Xi}}}}{=} \sum_{j \in [i+1]_0 } (x^R_j + z^R_j) - x^R_0 + \sum_{j\in [i]_0} (x^B_j + z^B_j)\\
	& \overset{\mathclap{\text{\eqref{eqn:sumq}}}}{\ge}  (\alpha +\eps/2)(t_i + t_{i+1}) -t_0,
\end{align*}
so the claim holds in this case.
		
Now, suppose that $t^{B}_{i+1} < t^{R}_{i+1}$.
By the choice of $t^R_{i+1}$, Lemma~\ref{lemma:usefulfacts}\ref{itm:u8} implies that $ |W_{t_i}^R| \ge |Z_{t_{i+1}}^B| = \sum_{j \in [i+1]_0} z^B_j$.
By a similar argument, $ |W_{t_i}^B| \ge \sum_{j \in [i]_0} z^R_j$.
Lemma~\ref{lemma:usefulfacts}\ref{itm:u3} and~\ref{itm:u6} imply that $X^B_{t_{i+1}} = X^{B}_{t_i}$ and so $x_{i+1}^B = 0$.
Hence 
\begin{align*}
	& |W^R_{t_i} \cup W^B_{t_i} \cup X^B_{t_i} \cup Z^B_{t_i} |
	 \ge |W^R_{t_i}| + |W^B_{t_i}| + |X^B_{t_i}|\\
		& \ge \sum_{j \in [i+1]_0 } z^B_j +  \sum_{j \in [i]_0} z^R_j + \sum_{j \in [i]_0} x^B_j 
		 \overset{\mathclap{\text{\eqref{eqn:Xi}}}}{=} \sum_{j \in [i+1]_0 } (x^B_j + z^B_j)  + \sum_{j\in [i]_0} (x^R_j + z^R_j) - x^R_0\\
		& \overset{\mathclap{\text{\eqref{eqn:sumq}}}}{\ge}  (\alpha +\eps/2)(t_i + t_{i+1}) -t_0. \qedhere
\end{align*}
\end{proofclaim}

Together with Lemma~\ref{lemma:usefulfacts}\ref{itm:u0} and~\eqref{eqn:YB+YR}, we have 
\begin{align*}
	t_i - |Z^R_{t_i}| - |X_{t_i}^R| &= |Y^B_{t_i} \cup Y^R_{t_i} |+ |W^R_{t_i} \cup W^B_{t_i} \cup X^B_{t_i} \cup Z^B_{t_i} |\\
	& \ge  (\alpha +\eps/2) (T_{i-1} + T_{i+1}) -t_0.
\end{align*}
Hence, \eqref{eqn:av1} and Lemma~\ref{lemma:usefulfacts}\ref{itm:u2} imply that 
\begin{align*}
	(1- \alpha ) (T_{i} - T_{i-1}) & = (1- \alpha ) t_i 
	\ge  c^B_{t_i} 
	\\
	& \ge 	(1 - 8 /\ell ) ( t_i - |Z^R_{t_i}| - |X_{t_i}^R| ) - \ell/2\\
& \ge  (\alpha + \eps/4)(T_{i-1} + T_{i+1}) -t_0 - \ell/2,\\
	0 & \ge (\alpha +\eps/4)  T_{i+1} - (1 -  \alpha) T_{i} + T_{i-1} - t_0 - \ell/2.
\end{align*}
Therefore, Lemma~\ref{lemma:differenceinequality} (and our choice of $\alpha$) implies 
\begin{align*}
	0 \le (1 -  \alpha)^2 - 4 (\alpha +\eps/4)  < 1 - 6 \alpha + \alpha^2 = 0 ,
\end{align*}
a contradiction.
\end{proof}

Now we are ready to prove Lemma~\ref{lemma:nopurplevertices}.

\begin{proof}[Proof of Lemma~\ref{lemma:nopurplevertices}]
	Let $\alpha \dfn (7 - \sqrt{17})/16$. %, so $(9 + \sqrt{17})/16 = 1 - \alpha$.
	Choose $\ell, k'_0 \in \NN$ such that $\ell$ is even, $k'_0 \ge k_0$ and
	\begin{align}
	0 < 1/k'_0 \ll 1/\ell \ll \eps, \alpha. \label{eq:truehierarchy}
	\end{align}
	Let $F^\ast_t$, $\Omega^\ast_t$, $\Gamma^\ast_t$, $A_t^{\ast}$, $\varphi_t$, $W^\ast_t$, $X^\ast_t$, $Y^\ast_t$, $Z^\ast_t$ be as defined by Algorithm~\ref{algo}.
	Lemma~\ref{lemma:usefulfacts}\ref{itm:u0} implies that for all $t \in \mathbb{N}$,
	\begin{align}
		t & = \sum_{\ast \in \{ R, B \}} |W^{\ast}_t|+|X^{\ast}_t|+|Y^{\ast}_t|+|Z^{\ast}_t|. \label{eq:totalattimet}
	\end{align}
	Lemma~\ref{lemma:usefulfacts}\ref{itm:u2} together with \eqref{eq:totalattimet} imply that for all $t \in \mathbb{N}$ (and a similar bound is true replacing $R$ by $B$):
	\begin{align}
		c^R_t & \ge (1 - 8/ \ell) (|W^R_t \cup Y^R_t| + |W^B_t \cup Y^B_t| + |X^R_t \cup Z^R_t|) - \ell/2. \label{eq:lowerboundcovered}
	\end{align}
	
	We might suppose that for all $t \ge k_0$ we have 
	\begin{align}
	c^R_t, c^B_t \leq (1 - \alpha - \eps)t, \label{eq:truenodensepathforest}
	\end{align} or else we are done.
	Together with Lemma~\ref{lemma:rho}\ref{itm:rho3} and  \eqref{eq:truehierarchy}, \begin{align}
		\rho^R_{t}(A^R_{t}) \ge \ell \text{ or } \rho^B_{t}(A^B_{t}) \ge \ell \qquad \forall t \ge k_0. \label{eq:alwaysavailable}
	\end{align}
	
	Without loss of generality, we may assume that $\rho^B_{k'_0}(A^B_{k'_0}) \ge \ell$.
	Define~$t_0$ to be the minimum $t > k'_0$ such that $\rho^B_t(A^B_t) < \ell$, which exists by Lemma~\ref{lemma:nomoreavailable} and \eqref{eq:truenodensepathforest}. % (since $\alpha > 3 - 2 \sqrt{2}$).
	Note that $\rho^R_{t_0}(A^R_{t_0}) \ge \ell$ by \eqref{eq:alwaysavailable}.
	Similarly, define $t_{1}$ to be the minimum $t > t_{0}$ such that $\rho^{R}_t(A^R_t) < \ell$.
	Now define $t_2$ to be the minimum $t > t_1$ such that $\rho^B_{t}(A^B_{t_1}) < \ell$.
	Note that~$t_2$ exists by Lemma~\ref{lemma:emptytheavailable} and \eqref{eq:truenodensepathforest}, and that $t_0 < t_1 < t_2$.
	
	Lemma~\ref{lemma:usefulfacts}\ref{itm:u2} and \eqref{eq:truenodensepathforest} imply that for all $\ast \in \{R, B\}$ and $i \in [2]$, \begin{align} 	|X^\ast_{t_i} \cup Z^\ast_{t_i} | & \ge (\alpha + \eps/2) t_i. \label{eq:unusedinti}
	\end{align}

	\begin{claim} \label{claim:constructionofhi}
		There exist 
		\begin{align}
			H^R \subseteq Y^R_{t_1} \cup W^R_{t_{1}} \text{ and } H^B \subseteq Y^B_{t_{1}} \cup W^B_{t_{1}} \label{eq:HsubsetofWY}
		\end{align}
		 such that \begin{align*}
		|H^R| & = |X^B_{t_{1}} \cup Z^B_{t_{1}}| - \ell, \\
		|H^B| & = |X^B_{t_{1}} \cup Z^B_{t_{1}}| + |X^R_{t_{2}} \cup Z^R_{t_{2}}| - \ell.
		\end{align*}
	\end{claim}
	
	\begin{proofclaim}
		For every $\ast \in \{ R, B \}$, consider $D^\ast_{t_1} \subseteq Y^\ast_{t_1} \cup W^\ast_{t_1}$ as given by Lemma~\ref{lemma:rho}.
		Note that $\rho^B_{t_0}(A^B_{t_0}) \leq \ell$.
		Then Lemma~\ref{lemma:rho}\ref{itm:rho1} implies
		\begin{align*}
		\rho^B_{t_1}(A^B_{t_1}) - \ell
		& \leq \rho^B_{t_1}(A^B_{1}) - \rho^B_{t_0}(A^B_{t_{0}}) 
		 \leq 2 |D^B_{t_1} \setminus D^B_{t_0}| + 2 |X^B_{t_1} \setminus X^B_{t_0}| \\
		& \leq 2 |D^B_{t_1}| + 2 |X^B_{t_1} \setminus X^B_{t_0}|.
		\end{align*}
		By the choice of $t_{0}$ and $t_{1}$, $\rho^{R}_{t'}(A^R_{t'}) \ge \ell$ for all $t_{0} \leq t' < t_1$.
		Therefore, Lemma~\ref{lemma:usefulfacts}\ref{itm:u3} and \ref{itm:u6} imply that $X^B_{t_1} \setminus X^B_{t_{0}} = \varnothing$.
		Hence, \begin{align}
		\rho^B_{t_1}(A^B_{t_1}) \leq 2 |D^B_{t_1}| + \ell. \label{eq:theinequalityaboveforrhoBt1}
		\end{align}
		
		Lemma~\ref{lemma:rho}\ref{itm:rho2} and \eqref{eq:theinequalityaboveforrhoBt1} together imply that
		\begin{align}
		2|D^R_{t_1}|
		& \ge 2 |D^B_{t_1}| + 2 |X^B_{t_1} \cup Z^B_{t_1}| - \rho^{B}_{t_1}(A^B_{t_1}) \ge 2 |X^B_{t_1} \cup Z^B_{t_1}| - \ell. \label{eq:theinequalityforDRtiabove}
		\end{align}
		Recall that $\rho^R_{t_1}(A^R_{t_1}) \leq \ell$ and $\rho^B_{t_2 - 1}(A^B_{t_1}) \ge \ell$.
		By Lemma~\ref{lemma:rho}\ref{itm:rho2strong},
		\begin{align*}
		2|D^B_{t_1}|
		& \ge 2 |D^R_{t_1}| + 2 |X^R_{t_2} \cup Z^R_{t_2}| - \rho^R_{t_1}(A^R_{t_1}) 
		 \ge 2 |D^R_{t_1}| + 2 |X^R_{t_2} \cup Z^R_{t_2}| - \ell \\
		& \stackrel{ \mathclap{\eqref{eq:theinequalityforDRtiabove}} }{\ge} 2 |X^B_{t_1} \cup Z^B_{t_1}| + 2 |X^R_{t_2} \cup Z^R_{t_2}| - 2 \ell.
		\end{align*}
		Thus $|D^B_{t_1}| \ge |X^B_{t_1} \cup Z^B_{t_1}| + |X^R_{t_2} \cup Z^R_{t_2}| - \ell$ and $|D^R_{t_1}| \ge |X^B_{t_1} \cup Z^B_{t_1}| - \ell$, which implies the existence of a set $H^\ast \subseteq D^\ast_{t_1} \subseteq Y^\ast_{t_1} \cup W^\ast_{t_1}$ of the desired size for every $\ast \in \{R, B\}$.	
	\end{proofclaim}
	
	Since $k'_0 \leq t_0 \leq t_1 \leq t_2$, we have $1/t_2, 1/{t_1} \ll 1/\ell \ll \alpha, \eps$.
	Let $H^R$ and $H^B$ be given by Claim~\ref{claim:constructionofhi}.
	Let \begin{align*}
	a & \dfn |X^B_{t_1} \cup Z^B_{t_1}|, & b & \dfn |X^R_{t_1} \cup Z^R_{t_1}|, \\
	c & \dfn |(X^B_{t_2} \cup Z^B_{t_2}) \setminus (X^B_{t_1} \cup Z^B_{t_1})|, & d & \dfn |(X^R_{t_2} \cup Z^R_{t_2}) \setminus (X^R_{t_1} \cup Z^R_{t_1})|.
	\end{align*}	
	Thus, $|H^R| = a - \ell$ and $|H^B| = a + b + d - \ell$.
%	Note that $t_1 \ge 3a + 2b + d - 2 \ell$.
	Let $\delta \dfn \eps/2$ and $\rho \dfn \alpha + \delta$.
	Since $\alpha = (7 - \sqrt{17})/16$ is the least real root of the polynomial $ 8x^2 - 7x + 1$ and $0 < \eps < 1/2$, it follows that $1 \leq 7 \rho -  8\rho^2$.
	
	Now we use the previous bounds to get \begin{align}
	1 - \alpha - \eps & \stackrel{ \mathclap{\eqref{eq:truenodensepathforest}} }{\ge} \; \frac{c^R_{t_1}}{t_1} \;
	\stackrel{ \mathclap{ \eqref{eq:lowerboundcovered}} }{\ge} \; \frac{ (1 - 8/\ell)( |W^R_{t_1} \cup Y^R_{t_1}| + |W^B_{t_1} \cup Y^B_{t_1}| + |X^R_{t_1} \cup Z^R_{t_1}|) - \ell / 2}{ t_1 }
	\nonumber \\
	& \stackrel{ \mathclap{\eqref{eq:totalattimet}} }{\ge} \; \frac{ |W^R_{t_1} \cup Y^R_{t_1}| + |W^B_{t_1} \cup Y^B_{t_1}| + |X^R_{t_1} \cup Z^R_{t_1}| - \ell / 2}{ |W^R_{t_1} \cup Y^R_{t_1}| + |W^B_{t_1} \cup Y^B_{t_1}| + | X^R_{t_1} \cup Z^R_{t_1}| + |X^B_{t_1} \cup Z^B_{t_1}| } - \frac{8}{\ell}
	\nonumber \\
	& \stackrel{ \mathclap{\eqref{eq:HsubsetofWY}} }{\ge} \; \frac{ |H^R| + |H^B| + |X^R_{t_1} \cup Z^R_{t_1}| - \ell / 2 }{ |H^R| + |H^B| + |X^R_{t_1} \cup Z^R_{t_1}| + |X^B_{t_1} \cup Z^B_{t_1}|} - \frac{8}{\ell} 
	\nonumber \\
%	& = \frac{2 |X^B_{t_i} \cup Z^B_{t_i}| + |X^R_{t_i} \cup Z^R_{t_i}| + |X^R_{t'_i} \cup Z^R_{t'_i}| - ( 8 t_i / \ell + 5 \ell / 2 ) }{3 |X^B_{t_i} \cup Z^B_{t_i}| + |X^R_{t_i} \cup Z^R_{t_i}| + |X^R_{t'_i} \cup Z^R_{t'_i}| - 2 \ell} \\
%	& \ge \frac{2 |X^B_{t_i} \cup Z^B_{t_i}| + |X^R_{t_i} \cup Z^R_{t_i}| + |X^R_{t'_i} \cup Z^R_{t'_i}|}{3 |X^B_{t_i} \cup Z^B_{t_i}| + |X^R_{t_i} \cup Z^R_{t_i}| + |X^R_{t'_i} \cup Z^R_{t'_i}|} - \frac{\eps}{2},
	& = \frac{2a + 2b + d - 5 \ell / 2}{3a + 2b + d - 2 \ell} - \frac{8}{\ell}
	 \ge \frac{2a + 2b + d}{3a + 2b + d} - \frac{\eps}{2}, \nonumber % \\
%	\rho & \leq \frac{a}{3 a + 2b + d}, \nonumber
	\end{align}
	where the last line follows from \eqref{eq:truehierarchy}, \eqref{eq:unusedinti} and $1/ t_1 \ll 1/\ell \ll \alpha, \eps$.
	Rearranging, we get $\rho \leq a/(3a+2b+d)$, and recalling that $1 \le 7 \rho -  8\rho^2$ we have
	\begin{align}
	3 a + 2b + d & \leq  (7 -  8\rho) a.
	\label{eq:keyequationA}
	\end{align}
	A similar argument (by estimating $c^B_{t_1} / t_1$) shows that \begin{align}
	3 a + 2b + d \leq (7 -  8\rho)b. \label{eq:keyequationB}
	\end{align}
	
	Next, we would like to estimate $c^B_{t_2} / t_2$ and $c^R_{t_2} / t_2$.
	By the choice of~$t_1$, Lemma~\ref{lemma:rho}\ref{itm:rho3} and \eqref{eq:truenodensepathforest}, \begin{align*}
	\rho^B_{t_1}(A^B_{t_1})
	& \ge 4(1 - 8/\ell) t_1 - 2 \ell - 2(c^R_{t_1} + c^B_{t_1}) - \rho^R_{t_1}(A^R_{t_1}) \\
	& \ge 4(1 - 8/\ell) t_1 - 2 \ell - 4(1 - \alpha - \eps) t_1 - \ell \\
	& \ge 4( \alpha + 2 \eps / 3) t_1,
	\end{align*} where the last inequality follows from \eqref{eq:truehierarchy}.
	Together with Lemma~\ref{lemma:usefulfacts}\ref{itm:u7} and the choice of $t_2$ we get \begin{align}
		2 |Y^B_{t_2} \setminus Y^B_{t_1}| & \ge \rho^B_{t_1}(A^B_{t_1}) - \rho^B_{t_2}(A^B_{t_1}) \ge 4( \alpha + 2 \eps / 3) t_1 - \ell 
	\nonumber \\
	& \ge 4 \rho (3a + 2b + d). \label{eq:boundyi}
	\end{align}
	Using Claim~\ref{claim:constructionofhi}, we get
	\begin{align*}
	1 - \alpha -\eps \stackrel{ \eqref{eq:truenodensepathforest} }{\ge} \frac{c^B_{t_2}}{t_2} \;
	& \stackrel{ \mathclap{ \eqref{eq:lowerboundcovered}} }{\ge} \; \frac{ (1 - 8/\ell)( |W^R_{t_2} \cup Y^R_{t_2}| + |W^B_{t_2} \cup Y^B_{t_2}| + |X^B_{t_2} \cup Z^B_{t_2}|) - \ell / 2}{ t_2 } \\
	& \stackrel{ \mathclap{ \eqref{eq:totalattimet}} }{\ge} \; \frac{ |W^R_{t_2} \cup Y^R_{t_2}| + |W^B_{t_2} \cup Y^B_{t_2}| + |X^B_{t_2} \cup Z^B_{t_2}| - \ell / 2}{ |W^R_{t_2} \cup Y^R_{t_2}| + |W^B_{t_2} \cup Y^B_{t_2}|+ |X^B_{t_2} \cup Z^B_{t_2}| + | X^R_{t_2} \cup Z^R_{t_2}|  } - \frac{8}{\ell} \\
	& \stackrel{ \mathclap{\eqref{eq:HsubsetofWY}} }{\ge} \frac{ |H^R| + |H^B| + |Y^B_{t_2} \setminus Y^B_{t_1}| + |X^B_{t_2} \cup Z^B_{t_2}| - \ell / 2 }{ |H^R| + |H^B| + |Y^B_{t_2} \setminus Y^B_{t_1}| + |X^B_{t_2} \cup Z^B_{t_2}| + | X^R_{t_2} \cup Z^R_{t_2}|} - \frac{8}{\ell} \\
	& \stackrel{\mathclap{\eqref{eq:boundyi}}}{\ge} \; \frac{2a + b + d + 2 \rho (3a + 2b + d) + a+c - 3 \ell / 2 }{2a + b + d + 2 \rho (3a + 2b + d) + a+c + b+d - 2 \ell} - \frac{8}{\ell} \\
	& \ge \frac{3a + b + c + d + 2 \rho(3a + 2b + d)}{3a + 2b + c + 2d + 2 \rho(3a + 2b + d)} - \frac{\eps}{2}, % \\
%	\rho & \leq \frac{b+d}{(1 + 2 \rho )(3 a + 2b + d) + c+d},
	\end{align*} 
	where the last inequality follows from \eqref{eq:truehierarchy}, \eqref{eq:unusedinti} and $1/t_2 \ll 1/\ell \ll \alpha, \eps$.
	Rearranging, we get $\rho \leq (b+d)/[(1 + 2 \rho )(3 a + 2b + d) + c+d]$.
	Recalling that $1 \le 7 \rho -  8\rho^2$, we get
	\begin{align}
	(1 + 2 \rho )(3 a + 2b + d) + c+d \leq (7  -  8\rho)(b+d). \label{eq:keyequationBD}
	\end{align}
	A similar argument (by estimating $c^R_{t_2} / t_2$) shows that \begin{align}
	(1 + 2 \rho )(3 a + 2b + d) + c+d \leq (7  -  8\rho)(a+c). \label{eq:keyequationAC}
	\end{align}

	By \eqref{eq:keyequationA}, \eqref{eq:keyequationB}, \eqref{eq:keyequationBD} and \eqref{eq:keyequationAC}, we deduce that $Ax \leq 0$, where $x = (a, b, c, d)^t$ and \[ A = \left[ \begin{array}{cccc}
	8 \rho - 4 & 2 & 0 & 1 \\
	3 & 8 \rho - 5 & 0 & 1 \\
	7 \rho - 2 & 1 + 2 \rho & 4 \rho - 3 & 1 + \rho \\
	3+6\rho & 12 \rho - 5 & 1 & 10 \rho - 5
	\end{array} \right]. \]
	Now consider the column vector $y = ( 7 - 12 \alpha, 2 - 4 \alpha, 1, 3 - 4 \alpha)^t$.
	Then $y \ge 0$ and $y^t A = ( ( 81 - 120 \alpha )\delta, (54 - 80 \alpha)\delta, 4 \delta, (31 - 40 \alpha ) \delta ) \ge (\delta, \delta, \delta, \delta) > 0$.
	Since $Ax \leq 0$ and $x,y \ge 0$, we get \begin{align*}
	0 \ge (y^t A) x \ge (\delta, \delta, \delta, \delta) x = \delta (a + b + c + d) > 0,
	\end{align*} a contradiction.
\end{proof}

\subsection*{Remark}

After the submission of this paper, we learned that Corsten, DeBiasio, Lamaison and Lang~\cite{CDLL18} have obtained an improved version of Theorem~\ref{theorem:main}.

\subsection*{Acknowledgements}

We thank an anonymous referee for their helpful suggestions.

\bibliography{monochromaticpathsetc}

\end{document}